\documentclass[11pt]{article}
\usepackage{graphicx,amsmath,amsthm,latexsym,amssymb,amstext}
\newtheorem{theorem}{Theorem}[section]
\newtheorem{definition}[theorem]{Definition}

\newtheorem{proposition}[theorem]{Proposition}
\newtheorem{lemma}[theorem]{Lemma}
\newtheorem{corollary}[theorem]{Corollary}
\newtheorem{remark}[theorem]{Remark}

\numberwithin{equation}{section}
\usepackage[body={14.5cm,22cm}, left=2.5cm, right=2.5cm, top=3cm]{geometry}
\usepackage[pdfstartview=FitH,
bookmarksnumbered=true,bookmarksopen=true,%
colorlinks=true,pdfborder=001,citecolor=blue,%
linkcolor=blue,urlcolor=blue]{hyperref}

\begin{document}

\title{\Large \bf Maximum principle for a stochastic delayed system involving terminal state constraints}

\author
{\textbf{Jiaqiang Wen}$^{1}$ and \textbf{Yufeng Shi}$^{1,2,}$\thanks{Corresponding author.
 E-mail addresses: jqwen@mail.sdu.edu.cn (J. Wen), yfshi@sdu.edu.cn (Y. Shi)} \vspace{3mm} \\
\normalsize{$^{1}$Institute for Financial Studies and School of Mathematics, }\\
\normalsize{Shandong University, Jinan 250100, China}\\
\normalsize{$^{2}$ School of Statistics, Shandong University of Finance and Economics, Jinan 250014, China}\\
}

\date{}

\renewcommand{\thefootnote}{\fnsymbol{footnote}}

\footnotetext[0]{The first and second authors are supported  by NNSF of China (Grant Nos. 11371226, 11071145, 11526205, 11626247 and 11231005), the Foundation for Innovative Research Groups of National Natural Science Foundation of China (Grant No. 11221061) and the 111 Project (Grant No. B12023).}

\maketitle

\begin{abstract}
We investigate a stochastic optimal control problem where the controlled system is depicted as a stochastic differential delayed equation;
however, at the terminal time, the state is constrained in a convex set.
We firstly introduce an equivalent backward delayed system depicted as a time-delayed backward stochastic differential equation.
Then a stochastic maximum principle is obtained by virtue of Ekeland's variational principle.
Finally, applications to a state constrained stochastic delayed linear-quadratic control model and a production-consumption choice problem are studied to illustrate the main obtained result.
\end{abstract}



\section{Introduction}

 In 1990, the nonlinear backward stochastic differential equation (BSDE in short) was introduced by
 Pardoux and Peng \cite{Peng}.
 Until now, it had applications in many fields,
  such as partial differential equation (see \cite{Peng3}),
 stochastic control (see \cite{ Karoui, Yong}) and mathematical finance (see \cite{Peng2}).
 Meanwhile, BSDE itself has been developed to many different branches, such as  BSDE with jumps (see \cite{Barles-Buckdan-Pardoux, Li-Wei-2014, Li-Wei-2015}), mean-field BSDE (see \cite{B-L-P-2009}), time-delayed BSDE  (see \cite{Li, Delong, Delong2}), anticipated BSDEs (see \cite{Yang,Wen}) and so on.
  A lot of works have been done for the control problem of such BSDEs.
  However, fewer works have been done on the control problems of stochastic delayed systems.

For a stochastic delayed system, Chen and Wu \cite{Wu} obtained a stochastic maximum principle by virtue of a duality between stochastic differential delayed equations (SDDEs in short) and anticipated BSDEs.
 {\O}ksendal, Sulem and Zhang \cite{Zhangt2011} studied the optimal control problems for SDDEs with jumps.
Yu \cite{Yuz2012} obtained a maximum principle for SDDEs with random coefficients.
A maximum principle of optimal control of SDDEs on infinite horizon
was proved in Agram, Haadem and {\O}ksendal \cite{Agram2013}.
Some other recent developments on stochastic delayed system can be found in
Huang, Li and Shi \cite{Huang2},
Meng and Shen \cite{Meng2015},
etc.

To the authors' knowledge,
there has been no result concerning the control problem of a stochastic delayed system with state constraints until now.
However, the state constraints of stochastic delayed systems indeed
exist in reality. In this paper, the stochastic control problem of a forward delayed system with terminal state constraint is studied.
The controlled system is depicted as the following SDDE:
\begin{equation}\label{36}
  \begin{cases}
   dX(t) = b(t,X(t),X(t-\delta),u(t)) dt  + \sigma (t,X(t),X(t-\delta),u(t)) dW(t), \ \ 0\leq t\leq T;\\
    X(t) = \eta(t), \ \ \ \ \ \ \ \ \ \ \ \ \ \ \ \ \ \ \ \ \ \ \ \ \ \ \ \ \ \ \ \ \ \ \ \ \ \ \ \ \ \ \ \ \ \ \ \ \ \ \ \ \ \ \ \ \ \ \ \ \ \ \ \ \ \ \ \ \ \ \ \ \ \ \
     -\delta\leq t\leq 0,
  \end{cases}
\end{equation}
where $X(T)\in K$, a.s., $K\in \mathbb{R}^{n}$ is a convex set.
However, there are two (main) difficulties in this study.
 The first one is that the control system (\ref{36}) is a delayed system,
 as stated in \cite{Wu}, which is more complex than the classical case.
Another difficulty is the terminal state constraint, which is a sample-wise constraint.
 As interpreted in Ji and Zhou \cite{Ji},
 the stochastic control involving sample-wise state constraints cannot be resolved by the classical theory.

Some recent developed results on state constraints (see \cite{Karoui, Ji2, Ji,Ji3, Wei,Aghayeva, Wei2016}) as well as the duality relation between time-advanced stochastic differential equations (SDEs, for short) and time-delayed BSDEs (see \cite{Li}) may help us to overcome the above mentioned difficulties.
Firstly, an equivalent backward formulation of stochastic delayed system (\ref{36}) is introduced,
where $X(T)$ is judged as a control variable.
 Meanwhile, the state constraint turns out to be a control constraint.
 However, such a treatment brings us both the advantage and the disadvantage.
The advantage is that, in the classical control theory,
to manage control constraint is easier than to manage state constraint.
The disadvantage is that the initial condition ($X(0)=\eta(0)$) now turns into an additional constraint.
To deal with the additional initial constraint, Ekeland's variational principle is used.

Note that the equivalent backward delayed system is described by a time-delayed BSDE,
so the adjoint equation of the time-delayed BSDE via duality relation is an anticipated SDE.
 Therefore, both the delayed system and the anticipated system are needed in our study.
As a routine, the variational procedure is made firstly.
Then, by virtue of Ekeland's variational principle, the variational inequality is got.
At last, the necessary condition is derived by applying the duality relationship between the backward delayed controlled system and the anticipated forward adjoint system.
There is a good thing  that the theory of BSDE and our assumption allow us to make the inverse transformation, so that
the optimal control process can be solved by the obtained optimal terminal control.
To make our conclusions be directly perceived, we also study two applications.
One of them is the stochastic delayed linear quadratic (LQ in short) control model.
Moreover, a production and consumption choice optimization problem (see \cite{Ivanov}) is also adapted to our case.

We organize this article as follows.
Some preliminary results about time-delayed BSDE and anticipated SDE are presented in Section 2.
In Section 3, the original control problem of
a forward delayed controlled system with terminal state constraint is formulated.
Then an equivalent transformation is made to get a backward delayed controlled system.
Moreover, a stochastic maximum principle is derived, which presents the required condition of the optimal terminal control.
In Section 4, two applications are given.

\section{Preliminaries}

Denote by $(\Omega,\mathcal{F},\mathbb{F},P)$ a probability space such that $\mathcal{F}_{0}$ includes all $P$-null elements of $\mathcal{F}$ and assume the filtration $\mathbb{F}=\{\mathcal{F}_t,t\geq 0\}$ is generated by a $d$-dimensional standard Brownian motion $W=\{W(t), t\geq 0\}$.
Let $T>0$. And $\delta> 0$ is a given finite time delay.
 We denote the following notations:
 \begin{itemize}
   \item [$\bullet$]  $L^{2}(\mathcal{F}_{t};\mathbb{R}^{n})=\{\xi: \Omega\rightarrow\mathbb{R}^n\big|\xi$ is $\mathcal{F}_{t}$-measurable, $E|\xi|^{2}$ $< \infty\}$;
   \item [$\bullet$] $L_{\mathbb{F}}^2(0,T;\mathbb{R}^{n})=\{\psi:\Omega\times[0,T]\rightarrow\mathbb{R}^{n}$\big|$\psi(\cdot)$ is  $\mathbb{F}$-measurable process, $\displaystyle E\int_0^T |\psi(t)|^{2} dt< \infty\}$.
 \end{itemize}
Similarly, we can define $L_{\mathbb{F}}^2(0,T;\mathbb{R}^{n\times d})$, $L_{\mathbb{F}}^2(-\delta,T;\mathbb{R}^{n})$ and $L_{\mathbb{F}}^2(0,T+\delta;\mathbb{R}^{n})$.

Now we recall some useful results for the study of the following sections.
Consider the following SDDE:
\begin{equation}\label{0}
  \begin{cases}
   dX(t) = b(t,X(t),X(t-\delta)) dt + \sigma (t,X(t),X(t-\delta)) dW(t), \ \ 0\leq t\leq T;\\
    X(t) = \eta(t), \ \ \ \ \ \ \ \ \ \ \ \ \ \ \ \ \  \ \ \ \ \ \ \ \ \ \ \ \ \ \ \ \ \ \ \ \ \ \ \ \ \ \ \ \ \ \ \ \ \ \ \ \ \ \ \ \ \ \ \ \ \
           -\delta\leq t\leq 0,
  \end{cases}
\end{equation}
where $\eta$ is a given continuous function, which represents the initial path of $X$, and
      $b:[0,T]\times \mathbb{R}^{n}\times \mathbb{R}^{n}\rightarrow \mathbb{R}^{n}$ and
 $\sigma:[0,T]\times \mathbb{R}^{n}\times \mathbb{R}^{n}\rightarrow \mathbb{R}^{n\times d}$
are given measurable functions satisfying the following condition:

\begin{itemize}
  \item [$\mathbf{(H2.1)}$] There exists a constant $D>0$ such that for all
    $t\in [0,T]$, $x,x',y,y'\in \mathbb{R}^{n}$,
    \begin{equation*}
      |b(t,x,y) - b(t,x',y')| + |\sigma(t,x,y) - \sigma(t,x',y')| \leq D(|x-x'| + |y-y'|);
    \end{equation*}
  \ \ \ \ \ \ \ \ \ \ \ $\sup\limits_{0\leq t\leq T} \big( |b(t,0,0)| + |\sigma(t,0,0)| \big)< +\infty.$
\end{itemize}
Then, from Theorem 2.2 in \cite{Wu}, under (H2.1), SDDE (\ref{0})has the unique adapted solution $X(\cdot)\in L_{\mathbb{F}}^2(-\delta,T;$ $\mathbb{R}^{n})$.

\medskip

For the time-delayed BSDE, we need the following assumption.

\begin{itemize}
  \item [$\mathbf{(H2.2)}$] Assume that
 $f:\Omega\times [0,T]\times \mathbb{R}^{n}\times \mathbb{R}^{n}\times \mathbb{R}^{n\times d}\rightarrow \mathbb{R}^{n}$ is $\mathbb{F}$-adapted and for every $y,y_{\delta},y',y_{\delta}'\in \mathbb{R}^{n}, z,z'\in\mathbb{R}^{n\times d}$
\begin{equation*}
   |f(t,y,y_{\delta},z)-f(t,y',y'_{\delta},z')|^{2}\leq C(|y-y'|^{2} +|y_{\delta}-y_{\delta}'|^{2} + |z-z'|^{2}),
\end{equation*}
where $C>0$ is a constant.
Moreover
 $\displaystyle E\int_0^T |f(t,0,0,0)|^{2} dt$ $ < +\infty$.
\end{itemize}

 The following is the well-posedness of  time-delayed BSDE.

\begin{proposition}\label{2}
Suppose $\xi\in L^{2}(\mathcal{F}_{T};\mathbb{R}^{n})$ and $\varphi(\cdot)$ is a given continuous function.
Then under (H2.2), for sufficiently small $\delta>0$, the following time-delayed BSDE
\begin{equation}\label{1}
  \begin{cases}
   -dY(t) = f(t,Y(t),Y(t-\delta),Z(t)) dt - Z(t) dW(t), \ \ 0\leq t\leq T; \\
     Y(T) = \xi, \ Y(t) = \varphi(t), \ \  \ \ \  \ \ \ \ \ \ \ \ \ \ \ \ \ \ \ \ \ \ \ \ \ \ \ \ \ \ \ \ \ -\delta\leq t< 0,
  \end{cases}
\end{equation}
has the unique adapted solution $(Y(\cdot),Z(\cdot))\in L_{\mathbb{F}}^2(-\delta,T;\mathbb{R}^{n})\times L_{\mathbb{F}}^2(0,T;\mathbb{R}^{n\times d})$,
and it satisfies the following estimate:
\begin{equation*}
  E\left[\sup_{0\leq t\leq T}|Y(t)|^{2} + \int_0^T |Z(t)|^{2} dt\right]
  \leq C E\left[|\xi|^{2} + \int_0^T |f(t,0,0,0)|^{2} dt\right],
\end{equation*}
with $C>0$.
Furthermore, if $(Y'(\cdot),Z'(\cdot))$ is the solution to (\ref{1}) with $\xi$ replaced by $\xi'$, then
\begin{equation*}
  E\left[\sup_{0\leq t\leq T}|Y(t)-Y'(t)|^{2} + \frac{1}{2}\int_0^T |Z(t)-Z'(t)|^{2} dt\right] \leq C E[|\xi-\xi'|^{2}].
\end{equation*}
\end{proposition}

As we stated in the part of Introduction, the anticipated SDE is necessary in our study. The following is the condition for anticipated SDE.

\begin{itemize}
  \item [$\mathbf{(H2.3)}$]
Suppose for each $t\in [0,T]$, $r\in [t,T+\delta]$,
$b:\Omega\times \mathbb{R}^{n}\times L^{2}(\mathcal{F}_{r};\mathbb{R}^{n}) \rightarrow L^{2}(\mathcal{F}_{t};\mathbb{R}^{n})$,
  $\sigma:\Omega\times \mathbb{R}^{n}\times L^{2}(\mathcal{F}_{r};\mathbb{R}^{n}) \rightarrow L^{2}(\mathcal{F}_{t};\mathbb{R}^{n\times d})$ with
\begin{align*}
|b(t,x, \varsigma_{t}) - b(t,x',\varsigma'_{t})| + |\sigma(t,x,\varsigma_{t}) - \sigma(t,x',\varsigma'_{t})|
\leq C\left(|x-x'| + E^{\mathcal{F}_{t}}[|\varsigma_{t}-\varsigma'_{t}|]\right),
\end{align*}
for every $t\in [0,T]$, $x,x'\in \mathbb{R}^{n}$,
$\varsigma(\cdot),\varsigma'(\cdot)\in L_{\mathbb{F}}^2(t,T+\delta;\mathbb{R}^{n})$, $r\in [t,T+\delta]$ with $C>0$.
Moreover
$\sup\limits_{0\leq t\leq T}\big(|b(t,0,0)| + |\sigma(t,0,0)|\big)< +\infty.$
\end{itemize}

\begin{proposition}\label{4}
Suppose $x_{0}\in \mathbb{R}^{n}$ and $\lambda(\cdot)\in L_{\mathbb{F}}^2(T,T+\delta;\mathbb{R}^{n})$ is a given $\mathbb{F}$-adapted process.
Assume (H2.3) holds. Then, if $\delta$ is sufficiently small, the anticipated SDE
\begin{equation}\label{3}
  \begin{cases}
    dX(t) = b(t,X(t),X(t+\delta)) dt + \sigma (t,X(t),X(t+\delta)) dW(t), \ \ 0\leq t\leq T;\\
     X(0) = x_{0}, \ X(t) = \lambda(t), \ \ \  \ \ \ \ \ \ \ \ \ \ \ \ \ \ \ \ \ \ \ \ \ \ \ \ \ \ \ \ \ \ \ \ \ \ \ \ \ \ \ \  T< t\leq T + \delta,
  \end{cases}
\end{equation}
has the unique adapted solution $X(\cdot)\in L_{\mathbb{F}}^2(0,T+\delta;\mathbb{R}^{n})$.
\end{proposition}

The above results  can be found in Delong and Imkeller \cite{Delong} and Chen and Huang \cite{Li}.
The following is the famous Ekeland's variational principle.

\begin{proposition}\label{9}
Suppose $(U,d(\cdot,\cdot))$ is a complete metric space with a function $F(\cdot):U\rightarrow \mathbb{R}$ is proper lower semi-continuous. Then, for every $v\in U$ and $\epsilon>0$ such that $F(v)\leq \inf\limits_{u\in U} F(u) +\epsilon$, there is $u_{\epsilon}\in U$ so that
\begin{flalign*}
\begin{split}
\ \ \ \ &(i) \  F(v_{\epsilon})\leq F(v),\\
\ \ \ \ &(ii) \ d(v,v_{\epsilon})\leq \epsilon, \\
\ \ \ \ &(iii) \ F(u) + \sqrt{\epsilon} d(u,v_{\epsilon})\geq F(v_{\epsilon}), \ \ \forall u\in U.
\end{split}&
\end{flalign*}
\end{proposition}

\section{Main result}

We study our main result in this part, i.e., a maximum principle about the optimal control of a stochastic delayed system involving terminal state constraint.
It should be pointed out that the time-delayed state of the controlled system is different from the case without delay.

\subsection{Problem formulation}
Let
\begin{equation*}
  \mathcal{U}_{ad} \equiv \{ u(\cdot)| u(\cdot)\in L_{\mathbb{F}}^{2}(0,T;\mathbb{R}^{n\times d}) \}
\end{equation*}
be the set of admissible controls.
For every given $u(\cdot)$, for the control system, we consider the past-dependent state $X(\cdot)$ depicted as
\begin{equation}\label{5}
  \begin{cases}
   dX(t) = b(t,X(t),X(t-\delta),u(t)) dt  + \sigma (t,X(t),X(t-\delta),u(t)) dW(t), \ \ 0\leq t\leq T;\\
    X(t) = \eta(t), \ \ \ \ \ \ \ \ \ \ \ \ \ \ \ \ \ \ \ \ \ \ \ \ \ \ \ \ \ \ \ \ \ \ \ \ \ \ \ \ \ \ \ \ \ \ \ \ \ \ \ \ \ \ \ \ \ \ \ \ \ \ \ \ \ \ \ \ \ \ \ \ \ \ -\delta\leq t\leq 0,
  \end{cases}
\end{equation}
where $\eta$ is a given continuous function,
      $b:[0,T]\times \mathbb{R}^{n}\times \mathbb{R}^{n}\times \mathbb{R}^{n\times d}\rightarrow \mathbb{R}^{n}$ and
 $\sigma:[0,T]\times \mathbb{R}^{n}\times \mathbb{R}^{n}\times \mathbb{R}^{n\times d}\rightarrow \mathbb{R}^{n\times d}$
are given measurable functions.
Define the cost function as follows:
\begin{equation*}
  J(u(\cdot))=E \left[ \int_0^T \widetilde{l}(t,X(t),u(t)) dt + \phi(X(T)) \right],
\end{equation*}
where $\widetilde{l}:[0,T]\times \mathbb{R}^{n}\times \mathbb{R}^{n\times d}\rightarrow \mathbb{R}^{n}$
and $\phi:\mathbb{R}^{n}\rightarrow \mathbb{R}^{n}$ are given measurable functions.
We give the following assumptions:

\begin{itemize}
  \item [$\mathbf{(H3.1)}$] The functions $b, \sigma, \widetilde{l},$ $\phi$ are all continuously differentiable in the arguments $(x,x',u)$, and their derivatives are all bounded.
  \item [$\mathbf{(H3.2)}$] Denote by $C(1 + |x| + |u|)$ and $C(1 + |x|)$
  the bounds of derivatives of $\widetilde{l}$ in its arguments $(x,u)$ and $\phi$ in its argument $x$, respectively.
\end{itemize}
Therefore, for every given $u(\cdot)\in \mathcal{U}_{ad}$, under assumptions (H3.1) and (H3.2), Eq. (\ref{5}) admits the unique adapted solution $X(\cdot)\in L_{\mathbb{F}}^2(-\delta,T;$ $\mathbb{R}^{n})$.

Denote by $K\in \mathbb{R}^{n}$ a given nonempty convex subset.
The goal of our control problem is to solve

\begin{equation*}{\mathbf{Problem \ A}: \ \ \ \ }
  \begin{cases}
    $Minimize$ \ \ \ \ J(u(\cdot)) \\
    $subject$ \ $to$ \ u(\cdot)\in \mathcal{U}_{ad}; \  X(T)\in K.
  \end{cases}
\end{equation*}

\subsection{Time-delayed backward formulation}

We now show an equivalent backward system of Problem A.
In order to do this, one additional assumption is needed:\\

$\mathbf{(H3.3)}$ There exists $\alpha>0$, and for each $t\in [0,T]$, $x,x'\in \mathbb{R}^{n}$ and $u_{1},u_{2}\in \mathbb{R}^{n\times d}$,
\begin{equation*}
 |\sigma(t,x,x',u_{1}) - \sigma(t,x,x',u_{2})| \geq \alpha|u_{1}-u_{2}|.
\end{equation*}

Note that (H3.1) and (H3.3) imply,
for every $(t,x,x')\in [0,T]\times \mathbb{R}^{n} \times \mathbb{R}^{n}$, that the following function
\begin{equation*}
  u \rightarrow \sigma(t,x,x',u)
\end{equation*}
is a bijection on $\mathbb{R}^{n\times d}$.
Hence, by letting $q\equiv \sigma(t,x,x',u)$,
we obtain that there is the inverse $\sigma^{-1}$ satisfying $u=\sigma^{-1}(t,x,x',q)$.
Then we can rewrite (\ref{5}) as
\begin{equation*}
  \begin{cases}
   -dX(t) = f(t,X(t),X(t-\delta),q(t)) dt - q(t) dW(t), \ \ 0\leq t\leq T;\\
    X(t) = \eta(t),  \ \ \ \ \ \ \ \ \ \ \ \ \ \ \ \ \ \ \  \ \ \ \ \ \ \ \ \ \ \ \ \ \ \ \ \ \ \ \ \ \ \ \ \ \ \ \ \ \ \ -\delta\leq t\leq 0,
  \end{cases}
\end{equation*}
where $f(t,x,x',q)=-b(t,x,x',\sigma^{-1}(t,x,x',q))$.

Note that $u \rightarrow \sigma(t,x,x',u)$ is a bijection, hence $q(\cdot)$ could be regarded as the control,
which is the crucial observation that encourages this method for working out Problem A.
Furthermore, by virtue of the theory of BSDE, choosing the terminal state $X(T)$ is equal to choosing $q(\cdot)$.
Therefore we innovate the following `controlled' system, which essentially is a time-delayed BSDE:
\begin{equation}\label{6}
  \begin{cases}
   -dX(t) = f(t,X(t),X(t-\delta),q(t)) dt - q(t) dW(t), \ \ 0\leq t\leq T;\\
    X(T) = \xi, \ \  X(t)=\eta(t), \ \ \ \ \ \ \ \ \ \ \ \ \ \ \ \ \ \ \ \ \ \ \ \ \ \ \ \ \ \ \ \ \ \ -\delta\leq t< 0,
  \end{cases}
\end{equation}
where now $\xi$ becomes the `control' and belongs to the following set:
\begin{equation*}
  U=\{ \xi| E|\xi|^{2}<\infty, \ \xi\in K, \ a.s. \}.
\end{equation*}
Moreover, here the equivalent cost function is
\begin{equation*}
  J(\xi):= E \left[ \int_0^T l(t,X(t),q(t)) dt + \phi(\xi) \right],
\end{equation*}
with $l(t,x,q)=\widetilde{l}(t,x,\sigma^{-1}(t,x,x',q))$.

Hence,
the original Problem A is equivalent to the following problem B:

\begin{equation}{\mathbf{Problem \ B}: \ \ \ \ }\label{7}
  \begin{cases}
    $Minimize$ \ \ \ \ J(\xi) \\
    $subject$ \ $to$ \ \ \ \xi\in U; \  X^{\xi}(0)=a,
  \end{cases}
\end{equation}
where $X^{\xi}(0)=a$ (we denote $a=\eta(0)$ in the following for simplicity) is the solution to
Eq. (\ref{6}) at the initial time 0 under $\xi$.

In control theory, it is well known that to solve the control constraint is easer than to solve the state constraint.
From now on, since Problem A is equivalent to Problem B, we concentrate on dealing with Problem B.
The benefit is that by virtue of $\xi$ becoming a control variable now,
a control constraint in Problem B replaces the state constraint in Problem A.

\begin{definition}
For $\xi\in U$ and $a\in \mathbb{R}^{n}$, if the solution to (\ref{6}) suits $X^{\xi}(0)=a$, then
we call the random variable $\xi$ feasible. For any given $a$, the collection of every feasible $\xi$ is denoted by
$\mathcal{N}(a)$.
Moreover, if $\xi^{*}\in U$ gets the minimum value of $J(\xi)$ over $\mathcal{N}(a)$,
we call $\xi^{*}$ optimal.
\end{definition}

\subsection{Variational equation}

In the following subsections, we denote the following notations for simplicity:
\[\begin{split}
      f(t) =& f(t,X(t),X(t-\delta),q(t)), \ \ f^{\rho}(t) = f(t,X^{\rho}(t),X^{\rho}(t-\delta),q^{\rho}(t)), \\
  f^{*}(t) =& f(t,X^{*}(t),X^{*}(t-\delta),q^{*}(t)), \ \
  f^{*}_{\varphi}(t) = f_{\varphi}(t,X^{*}(t),X^{*}(t-\delta),q^{*}(t)),\\
\end{split}\]
where $f_{\varphi}$ denotes the partial derivative of f at $\varphi$ with $\varphi=x,x_{\delta},q$, respectively.
In $U$, for $\xi^{1},\xi^{2}\in U$, define a metric by
\begin{equation*}
  d(\xi^{1},\xi^{2}) := (E|\xi^{1}-\xi^{2}|^{2})^{\frac{1}{2}}.
\end{equation*}
Apparently, $(U,d(\cdot,\cdot))$ becomes a complete metric space.
Suppose $\xi^{*}$ is optimal and, associated with $\xi^{*}$,
 the pair $(X^{*}(\cdot),$ $q^{*}(\cdot))$ is the corresponding state processes of Eq. (\ref{6}).
Because $U$ is convex, for every $\xi$, the following variational control $\xi^{\rho}$ is also in $U$:
\begin{equation*}
  \xi^{\rho}:= \xi^{*} + \rho(\xi - \xi^{*}), \ \ 0\leq \rho \leq 1.
\end{equation*}
Denote the solution to Eq. (\ref{6}) associated with $\xi=\xi^{\rho}$ by $(X^{\rho}(\cdot),q^{\rho}(\cdot))$.
And denote by $(\widehat{X}(\cdot),\widehat{q}(\cdot))$ the solution of following variational equation:
\begin{equation}\label{8}
  \begin{cases}
   -d\widehat{X}(t) = [f^{*}_{x}(t)\widehat{X}(t)
     + f^{*}_{x_{\delta}}(t)\widehat{X}(t-\delta) + f^{*}_{q}(t)\widehat{q}(t)]dt - \widehat{q}(t) dW(t), \ 0\leq t\leq T;\\
    \widehat{X}(T) = \xi-\xi^{*}, \ \widehat{X}(t)=0, \ \ \ \ \ \ \ \ \ \ \ \ \ \ \ \ \ \ \ \ \ \ \ \ \ \ \ \ \ \ \ \ \ \ \ \ \ \ \ \ \ \ \ \ \ \ \ \ \ \ \ \ -\delta\leq t< 0.
  \end{cases}
\end{equation}

\begin{remark}
It is easy to know that (\ref{8}) is a linear time-delayed BSDE.
By Proposition \ref{2}, under conditions (H3.1)-(H3.3), Eq. (\ref{8}) has a unique adapted solution in $L_{\mathbb{F}}^2(-\delta,T;\mathbb{R}^{n})\times L_{\mathbb{F}}^2(0,T;\mathbb{R}^{n\times d})$.
\end{remark}

\begin{lemma}\label{11}
Under assumptions (H3.1)-(H3.3), one has
\[\begin{split}
 \lim_{\rho\rightarrow 0}\sup_{0\leq t\leq T} E|\widetilde{X}^{\rho}(t)|^{2} = 0, \\
 \lim_{\rho\rightarrow 0} E\int_0^T |\widetilde{q}^{\rho}(t)|^{2} dt = 0,
\end{split}\]
where
\begin{equation*}
  \widetilde{X}^{\rho}(t) = \frac{X^{\rho}(t) - X^{*}(t)}{\rho} - \widehat{X}(t), \ \ \ \
  \widetilde{q}^{\rho}(t) = \frac{q^{\rho}(t) - q^{*}(t)}{\rho} - \widehat{q}(t).
\end{equation*}
\end{lemma}

\begin{remark}
Since the proof of Lemma \ref{11} above is the same as that of Lemma 3.1 in Chen and Huang \cite{Li},
for simplicity of presentation, we only present the main result and omit the detailed proof.
In fact, it is straightforward to prove Lemma \ref{11} by applying Proposition \ref{2}, Taylor expansion and the Lebesgue dominated convergence theorem.
\end{remark}

\subsection{Variational inequality}

We solve the initial constraint $X^{\xi}(0)=a$ and obtain a variational inequality in this subsection.

Given the optimal $\xi^{*}$,
for a constat $\varepsilon>0$, define $F\varepsilon(\cdot):U\rightarrow \mathbb{R}$ as follows:
\begin{equation*}
  F_{\varepsilon}(\xi)= \left\{ |X^{\xi}(0)-a|^{2} + \big(\max (0, \phi(\xi) - \phi(\xi^{*}) + \varepsilon)\big)^{2} \right\}^{\frac{1}{2}}.
\end{equation*}

\begin{remark}
One can test that the functions $|X^{\xi}(0)-a|^{2}$ and $\phi(\xi)$ are both continuous in their argument  $\xi$.
Hence, $F_{\varepsilon}$, defined on $U$, is also a continuous function in its argument $\xi$.
\end{remark}

\begin{theorem}
Under assumptions (H3.1)-(H3.3), suppose that $\xi^{*}$ is an optimal solution to Problem B,
so we have $h_{0}\in \mathbb{R}^{+}$ and $h_{1}\in \mathbb{R}^{n}$ satisfying $|h_{0}|+|h_{1}|\neq 0$,
so that for every $\xi\in U$, we have the following variational inequality:
\begin{equation}\label{10}
  \big<h_{1},\widehat{X}(0)\big> + h_{0} \big<\phi_{x}(\xi^{*}), \xi - \xi^{*}\big> \geq 0,
\end{equation}
where $\widehat{X}(0)$ is the solution to Eq. (\ref{8}) at time 0.
\end{theorem}

\begin{proof}
We can check the following properties by the definition,
\begin{flalign*}
\begin{split}
\ \ \ \ \ \ \ \ \ &F_{\varepsilon}(\xi^{*})=\varepsilon;\\
\ \ \ \ \ \ \ \ \ &F_{\varepsilon}(\xi)> 0,  \ \ \forall \xi\in U;\\
\ \ \ \ \ \ \ \ \ &F_{\varepsilon}(\xi^{*})\leq \inf\limits_{\xi\in U} F_{\varepsilon}(\xi) +\varepsilon.
\end{split}&
\end{flalign*}
Therefore, from Proposition \ref{9} (Ekeland's variational principle), there exists $\xi^{\varepsilon}\in U$ satisfying:
\begin{flalign*}
\begin{split}
\ \ \ \ \ \ \ \ \ &(i) \ F_{\varepsilon}(\xi^{\varepsilon})\leq F_{\varepsilon}(\xi^{*});\\
\ \ \ \ \ \ \ \ \ &(ii) \ d(\xi^{\varepsilon},\xi^{*})\leq \sqrt{\varepsilon};\\
\ \ \ \ \ \ \ \ \ &(iii) \ F_{\varepsilon}(\xi)+\sqrt{\varepsilon}d(\xi,\xi^{\varepsilon})\geq F_{\varepsilon}(\xi^{\varepsilon}), \ \ \forall\xi\in U.
\end{split}&
\end{flalign*}
For every $\xi\in U$, denote $\xi_{\rho}^{\varepsilon}:= \xi^{\varepsilon}+ \rho(\xi - \xi^{\varepsilon}), \ 0\leq \rho\leq 1.$
Let $(X_{\rho}^{\varepsilon}(\cdot),q_{\rho}^{\varepsilon}(\cdot))$ (resp. $X^{\varepsilon}(\cdot),q^{\varepsilon}(\cdot)$) be the solution to (\ref{6}) under $\xi_{\rho}^{\varepsilon}$ (resp. $\xi^{\varepsilon}$), and let $(\widehat{X}^{\varepsilon}(\cdot),\widehat{q}^{\varepsilon}(\cdot))$ be the solution to (\ref{8}) when $\xi^{\varepsilon}$ is replaced by $\xi^{*}$.
 Hence, applying the item (iii) above, one obtains
\begin{equation}\label{12}
  F_{\varepsilon}(\xi_{\rho}^{\varepsilon}) - F_{\varepsilon}(\xi^{\varepsilon}) +\sqrt{\varepsilon}d(\xi_{\rho}^{\varepsilon},\xi^{\varepsilon})\geq 0.
\end{equation}
On the other hand, similar to Lemma \ref{11}, one concludes
\begin{equation*}
 \lim_{\rho\rightarrow 0}\sup_{0\leq t\leq T}
      E \left[ \rho^{-1}\big( X_{\rho}^{\varepsilon}(t) - X^{\varepsilon}(t)\big) - \widehat{X}^{\varepsilon}(t) \right]= 0.
\end{equation*}
Thus
\begin{equation*}
   X_{\rho}^{\varepsilon}(0) - X^{\varepsilon}(0)= \rho\widehat{X}^{\varepsilon}(0) + o(\rho),
\end{equation*}
which leads to the following expansion:
\begin{equation*}
   \left|X_{\rho}^{\varepsilon}(0)-a\right|^{2} - \left|X^{\varepsilon}(0)-a\right|^{2}
   = 2\rho \big<X^{\varepsilon}(0)-a, \widehat{X}^{\varepsilon}(0)\big> + o(\rho).
\end{equation*}
Moreover,
\begin{equation*}
      |\phi(\xi_{\rho}^{\varepsilon}) - \phi(\xi^{*}) + \varepsilon|^{2}
    - |\phi(\xi^{\varepsilon}) - \phi(\xi^{*}) + \varepsilon|^{2}
   = 2\rho [\phi(\xi^{\varepsilon}) - \phi(\xi^{*}) + \varepsilon] \big<\phi_{x}(\xi^{\varepsilon}),\xi-\xi^{\varepsilon}\big> + o(\rho).
\end{equation*}
In the next, we study the following two cases for given $\varepsilon>0$.

Case 1: There is $\rho_{0}>0$ so that for every $\rho\in (0,\rho_{0})$,
\begin{equation*}
  \phi(\xi_{\rho}^{\varepsilon}) -  \phi(\xi^{*}) + \varepsilon \geq 0.
\end{equation*}
We see that
\[\begin{split}
  &\lim_{\rho\rightarrow 0}\frac{F_{\varepsilon}(\xi_{\rho}^{\varepsilon}) - F_{\varepsilon}(\xi^{\varepsilon})}{\rho}\\
 =&\lim_{\rho\rightarrow 0}\frac{1}{F_{\varepsilon}(\xi_{\rho}^{\varepsilon})
   + F_{\varepsilon}(\xi^{\varepsilon})}\frac{F_{\varepsilon}^{2}(\xi_{\rho}^{\varepsilon}) - F_{\varepsilon}^{2}(\xi^{\varepsilon})}{\rho}\\
 =&\frac{1}{F_{\varepsilon}(\xi^{\varepsilon})} \left\{ \big<X^{\varepsilon}(0)-a,\widehat{X}^{\varepsilon}(0)\big>
   +\big[\phi(\xi^{\varepsilon}) - \phi(\xi^{*}) + \varepsilon\big] \cdot \big<\phi_{x}(\xi^{\varepsilon}),\xi-\xi^{\varepsilon}\big>\right\}.
\end{split}\]
Now, using $\rho$ to split (\ref{12}) and letting $\rho$ to 0, one has
\begin{equation}\label{13}
  h_{0}^{\varepsilon} \big<\phi_{x}(\xi^{\varepsilon}),\xi-\xi^{\varepsilon}\big> + \big< h_{1}^{\varepsilon},\widehat{X}^{\varepsilon}(0) \big>
  \geq -\sqrt{\varepsilon} [E|\xi - \xi^{\varepsilon}|^{2}]^{\frac{1}{2}},
\end{equation}
where
\begin{equation*}
  h_{0}^{\varepsilon} = \frac{1}{F_{\varepsilon}(\xi^{\varepsilon})} \cdot [\phi(\xi^{\varepsilon}) - \phi(\xi^{*}) + \varepsilon\big]\geq 0, \ \ \ \
  h_{1}^{\varepsilon} = \frac{1}{F_{\varepsilon}(\xi^{\varepsilon})} \big<X^{\varepsilon}(0)-a \big>.
\end{equation*}
Case 2: There is a positive series $\{\rho_{n}\}$ satisfying $\rho_{n}\rightarrow 0$, so that
\begin{equation*}
  \phi(\xi_{\rho_{n}}^{\varepsilon}) -  \phi(\xi^{*}) + \varepsilon \leq 0.
\end{equation*}
From the definition of $F_{\varepsilon}$, for large $n$,
 $F_{\varepsilon}(\xi_{\rho_{n}}^{\varepsilon})= \left\{ |X_{\rho_{n}}^{\varepsilon}(0)-a|^{2} \right\}^{\frac{1}{2}}$.
Owing to the continuity of $F_{\varepsilon}(\cdot)$, one has
$F_{\varepsilon}(\xi^{\varepsilon})= \left\{ |X^{\varepsilon}(0)-a|^{2} \right\}^{\frac{1}{2}}$.

Moreover,
\begin{align*}
  \lim_{n\rightarrow 0}\frac{F_{\varepsilon}(\xi_{\rho_{n}}^{\varepsilon}) - F_{\varepsilon}(\xi^{\varepsilon})}{\rho}
 =&\lim_{n\rightarrow 0}\frac{1}{F_{\varepsilon}(\xi_{\rho_{n}}^{\varepsilon})
   + F_{\varepsilon}(\xi^{\varepsilon})} \frac{F_{\varepsilon}^{2}(\xi_{\rho_{n}}^{\varepsilon}) - F_{\varepsilon}^{2}(\xi^{\varepsilon})}{\rho}\\
=&\frac{\big<X^{\varepsilon}(0)-a,\widehat{X}^{\varepsilon}(0) \big>}{F_{\varepsilon}(\xi^{\varepsilon})}.
\end{align*}
From (\ref{12}), the same as in Case 1,
\begin{equation}\label{14}
  \big< h_{1}^{\varepsilon},\widehat{X}^{\varepsilon}(0) \big> \geq -\sqrt{\varepsilon} [E|\xi - \xi^{\varepsilon}|^{2}]^{\frac{1}{2}},
\end{equation}
where
\begin{equation*}
  h_{0}^{\varepsilon} = 0, \ \ \ \
  h_{1}^{\varepsilon} = \frac{1}{F_{\varepsilon}(\xi^{\varepsilon})} \big<X^{\varepsilon}(0)-a \big>.
\end{equation*}
For both cases, in summary, from the definition of $F_{\varepsilon}(\cdot)$, one has $ h_{0}^{\varepsilon}\geq 0$ and
\begin{equation*}
  |h_{0}^{\varepsilon}|^{2} +  |h_{1}^{\varepsilon}|^{2} =1.
\end{equation*}
Therefore, there exists a convergent subsequence of $(h_{0}^{\varepsilon},h_{1}^{\varepsilon})$ whose limit is denoted by $(h_{0},h_{1})$.

Due to $d(\xi^{\varepsilon},\xi^{*})\leq \sqrt{\varepsilon}$, we have
$\xi^{\varepsilon}\rightarrow  \xi^{*}, \ \ as \ \ \varepsilon\rightarrow 0.$
Then from the estimate of Proposition \ref{2}, we see that $\widehat{X}^{\varepsilon}(0)\rightarrow \widehat{X}(0)$ as $\varepsilon\rightarrow 0$.
Thus (\ref{10}) holds. The desired result is proved now.
\end{proof}

By using similar analysis, when $l(t,x,q)\neq0$, the following variational inequality can be obtained.
\begin{theorem}
Let (H3.1)-(H3.3) hold. Suppose that $\xi^{*}$ is an optimal solution of Problem B.
Then we have $h_{0}\in \mathbb{R}^{+}$, $h_{1}\in \mathbb{R}^{n}$ satisfying $|h_{0}|+|h_{1}|\neq 0$,
so that for every $\xi\in U$,
we have the variational inequality:
\begin{equation}\label{16}
  \big<h_{1},\widehat{X}(0)\big> + h_{0} \big<\phi_{x}(\xi^{*}), \xi - \xi^{*}\big>
   + h_{0}\int_0^T \big<l_{x}^{*}(t), \widehat{X}(t)\big> dt +  h_{0}\int_0^T \big<l_{q}^{*}(t), \widehat{q}(t)\big> dt \geq 0,
\end{equation}
where $l_{\varphi}^{*}(t)=l_{\varphi}(t,X^{*}(t),q^{*}(t))$ denotes the partial derivative of $l^{*}$ at $\varphi$ with
$\varphi= x,q,$ respectively, and $(\widehat{X}(\cdot),\widehat{q}(\cdot))$ is the solution to variation equation (\ref{8}).
\end{theorem}

\subsection{Maximum principle}

For the sake of establishing the maximum principle, in this part,
as the dual equation of Eq. (\ref{8}), the following equation is introduced:
\begin{equation}\label{17}
  \begin{cases}
   dm(t) = \left\{f^{*}_{x}(t)^{T}m(t) + E^{\mathcal{F}_{t}}\big[(f^{*}_{x_{\delta}}|_{t+\delta})^{T}m(t+\delta)\big] + h_{0}l^{*}_{x}(t)\right\} dt\\
 \ \ \ \ \ \ \ \ \ \ \ \   + [f^{*}_{q}(t)^{T}m(t) + h_{0}l^{*}_{q}(t)] dW(t), \ \ \ \ \ \ \ \ \ \ \ \ \ \ \
    0\leq t\leq T;\\
    m(0) =h_{1}, \ \ m(t)=0, \ \ \ \ \ \ \ \ \ \ \ \ \ \ \ \ \ \ \ \ \ \ \ \ \ \ \ \ \ \ \ \ \ \ \ \ \  T< t\leq T+\delta.
  \end{cases}
\end{equation}

\begin{remark}
In Eq. (\ref{17}), $f^{*}_{x_{\delta}}|_{t+\delta}$ represents the value of $f^{*}_{x_{\delta}}$ when $t$ is replaced by $t+\delta$.
\end{remark}

\begin{remark}
It is easy to see that (\ref{17}) is a linear time-advanced SDE.
 By Proposition \ref{4}, under conditions (H3.1)-(H3.3),
 Eq. (\ref{17}) admits the unique adapted solution in $L_{\mathbb{F}}^2(0,T+\delta;\mathbb{R}^{n})$.
\end{remark}

\begin{theorem}\label{18}
Let (H3.1)-(H3.3) hold.
If $\xi^{*}$ is optimal to Problem B with $(X^{*}(\cdot),$ $q^{*}(\cdot))$ being the corresponding state of Eq. (\ref{6}),
then we have $h_{0}\in \mathbb{R}^{+}$ and $h_{1}\in \mathbb{R}^{n}$ satisfying $|h_{0}|+|h_{1}|\neq 0$, so that for every $\eta\in U$,
\begin{equation}
 \big<m(T) + h_{0}\phi_{x}(\xi^{*}), \eta - \xi^{*}\big> \geq 0, \ \ a.s.
\end{equation}
where $m(\cdot)$ is the solution of Eq. (\ref{17}).
\end{theorem}

\begin{proof}
By using It\^{o}'s formula to $ \big<m(t),\widehat{X}(t)\big>$, we obtain
\begin{equation*}
\begin{split}
d\big<m(t),\widehat{X}(t)\big>
 =&-m(t)[f^{*}_{x}(t)\widehat{X}(t) + f^{*}_{x_{\delta}}(t)\widehat{X}(t-\delta) + f^{*}_{q}(t)\widehat{q}(t)] dt\\
  &+\widehat{X}(t)\left\{f^{*}_{x}(t)^{T}m(t) + E^{\mathcal{F}_{t}}\big[(f^{*}_{x_{\delta}}|_{t+\delta})^{T}m(t+\delta)\big] + h_{0}l^{*}_{x}(t)\right\}
    dt\\
  &+\widehat{q}(t)[f^{*}_{q}(t)^{T}m(t) + h_{0}l^{*}_{q}(t)] dt + \{\cdot \cdot \cdot\} dW(t)\\
 =&\left\{E^{\mathcal{F}_{t}}\big[(f^{*}_{x_{\delta}}|_{t+\delta})^{T}m(t+\delta)\big]\widehat{X}(t) -f^{*}_{x_{\delta}}(t)m(t)\widehat{X}(t-\delta)
   \right\} dt\\
  &+h_{0}[l^{*}_{x}(t)\widehat{X}(t) + l^{*}_{q}(t)\widehat{q}(t)] dt + \{\cdot \cdot \cdot\} dW(t).
\end{split}
\end{equation*}
Therefore,
\begin{equation*}
  E[\big<m(T),\widehat{X}(T)\big> - \big<m(0),\widehat{X}(0)\big>] = \Delta_{1} + \Delta_{2},
\end{equation*}
with
\[\begin{split}
&\Delta_{1}=E\int_0^T \big[(f^{*}_{x_{\delta}}|_{t+\delta})^{T}m(t+\delta)\widehat{X}(t) -f^{*}_{x_{\delta}}(t)m(t)\widehat{X}(t-\delta)\big]  dt,\\
&\Delta_{2}=h_{0}E\int_0^T \big[l^{*}_{x}(t)\widehat{X}(t) + l^{*}_{q}(t)\widehat{q}(t)\big]  dt.
\end{split}\]
Paying attention to the terminal and initial conditions, one derives
\[\begin{split}
\Delta_{1}=&E\int_0^T (f^{*}_{x_{\delta}}|_{t+\delta})^{T}m(t+\delta)\widehat{X}(t) dt - E\int_0^T f^{*}_{x_{\delta}}(t)m(t)\widehat{X}(t-\delta) dt\\
=&E\int_T^{T+\delta} f^{*}_{x_{\delta}}(t)^{T}m(t)\widehat{X}(t-\delta) dt-E\int_0^{\delta}f^{*}_{x_{\delta}}(t)m(t)\widehat{X}(t-\delta) dt\\
=&0.
\end{split}\]
Hence
\begin{align*}
 &E[\big<m(T)+h_{0}\phi_{x}(\xi^{*}),\xi - \xi^{*}\big>]\\
 =&E\bigg[\big<h_{1},\widehat{X}(0)\big> + h_{0} \big<\phi_{x}(\xi^{*}), \xi - \xi^{*}\big>
  + h_{0}\int_0^T \big<l_{x}^{*}(t), \widehat{X}(t)\big> dt +  h_{0}\int_0^T \big<l_{q}^{*}(t), \widehat{q}(t)\big> dt \bigg]\geq 0.
\end{align*}
From the arbitrariness of $\xi\in U$, for every $\eta\in U$, we have
\begin{equation*}
 \big<m(T) + h_{0}\phi_{x}(\xi^{*}), \eta - \xi^{*}\big> \geq 0, \ \ a.s.
\end{equation*}
\end{proof}

Now, we let $\partial K$ represent the boundary of $K$, and denote
$$\Omega_{0}:=\{ \omega\in\Omega| \xi^{*}\in \partial K \}.$$
According to Theorem \ref{18}, we directly deduce the following result.

\begin{corollary}
Assume that the assumptions in Theorem \ref{18} hold, then for each $\eta\in K$, we have
\[\begin{split}
\big<&m(T) + h_{0}\phi_{x}(\xi^{*}), \eta - \xi^{*}\big> \geq 0, \ \ a.s. \  on \ \Omega_{0};\\
     &m(T) + h_{0}\phi_{x}(\xi^{*}) = 0,  \ \ a.s. \  on \ \Omega_{0}^{c}.
\end{split}\]
\end{corollary}

\begin{remark} By the above study, for the optimal terminal control $\xi^*$, we obtain the necessary condition.
Note that the previous transformation process and (H3.3) allow us to make the inverse transformation.
Therefore, the characterization of the optimal control process $u^*(\cdot)$ can be derived by the obtained stochastic maximum principle of
the optimal terminal control $\xi^*$.
\end{remark}

\section{Applications of the main result}
As stated in the section of Introduction, we study two applications of the main result established above in this section.

\subsection{Stochastic delayed LQ control involving terminal state constraints}

Stochastic delayed LQ control problem involving terminal state constraints is considered in this subsection.
In order to simplify the presentation, we focus on the case $d=n=1$.
For the higher dimensional situation, one can deal with it in a similar method without substantial difficulty.

Consider the following state equation:
\begin{equation}\label{22}
  \begin{cases}
   dX(t) = \big[A_{1}X(t)+A_{2}X(t-\delta)+A_{3}u(t)\big] dt \\
           \ \ \ \ \ \ \ \ \ \ \ + \big[B_{1}X(t)+B_{2}X(t-\delta)+B_{3}u(t)\big] dW(t), \ \ 0\leq t\leq T;\\
    X(t) = \eta(t),  \ \ \ \ \ \ \ \ \ \ \ \ \ \ \ \ \ \ \ \ \ \ \ \ \ \ \ \ \ \ \ \ \ \ \ \ \ \ \ \ \ \ \ \ \ \ \ \ \ \  -\delta\leq t\leq 0,
  \end{cases}
\end{equation}
with $A_{i},B_{i}\in \mathbb{R}, \ i=1,2,3$.

Next, we investigate the cost function independent of the running cost without loss of generality.
Therefore, subject to $u(\cdot)\in \mathcal{U}_{ad}, \ X(T)\in \mathbb{R}^{+}, a.s.$,
the goal is to minimize the following cost function:
\begin{equation}\label{23}
  J(u(\cdot))=\frac{1}{2} E[X(T)^2].
\end{equation}
Without doubt, Eq. (\ref{23}) is an extraordinary example of Problem A with
\begin{equation*}
  b(t,x,x',u)=A_{1}x+A_{2}x'+A_{3}u, \ \ \ \  \sigma(t,x,x',u)=B_{1}x+B_{2}x'+B_{3}u.
\end{equation*}
Now we give the backward formulation of problem (\ref{23}).
Denote
 \begin{equation*}
 \overline{A}_{1}=A_{3}B_{1}B_{3}^{-1} - A_{1}, \ \ \overline{A}_{2}=A_{3}B_{2}B_{3}^{-1}-A_{2}, \ \ \overline{A}_{3}=-A_{3}B_{3}^{-1}.
 \end{equation*}
 Then Eq. (\ref{22}) becomes
\begin{equation}\label{24}
  \begin{cases}
   -dX(t) = \big(\overline{A}_{1}X(t)+\overline{A}_{2}X(t-\delta)+\overline{A}_{3}u(t)\big) dt - q(t) dW(t), \ 0\leq t\leq T;\\
    X(T) = \xi, \ \  X(t)=\eta(t),  \ \ \ \ \ \ \ \ \ \ \ \ \ \ \ \ \ \ \ \ \ \ \ \ \ \ \ \ \ \ \ \ \ \ \ \ \ \ \ \ \ \ \ \ -\delta\leq t< 0,
  \end{cases}
\end{equation}
and we can rewrite problem (\ref{23}) as follows:
\begin{equation}\label{25}
  \begin{cases}
    $Minimize$ \ \ \ \ J(\xi) \\
    $subject$ \ $to$ \ \ \ \xi\in U; \  X^{\xi}(0)=a,
  \end{cases}
\end{equation}
where
\begin{equation*}
  U=\{ \xi| E|\xi|^{2}<\infty, \ \xi\in \mathbb{R}^{+}, \ a.s. \}.
\end{equation*}
By Theorem \ref{18}, if $\xi^{*}$ is optimal, we have $h_{0}\in \mathbb{R}^{+}$ and $h_{1}\in \mathbb{R}$ satisfying $|h_{0}|+|h_{1}|\neq 0$, so that  $\forall \ \eta\in U$,
\begin{equation}
 \big<m(T) + h_{0}\xi^{*}, \eta - \xi^{*}\big> \geq 0, \ \ a.s.,
\end{equation}
in which $m(\cdot)$ is the solution to the following adjoint equation:
\begin{equation}\label{26}
  \begin{cases}
   dm(t) = \left(\overline{A}_{1}m(t) + \overline{A}_{2}E^{\mathcal{F}_{t}}\big[m(t+\delta)\big]\right) dt + \overline{A}_{3}m(t) dW(t),  \ \ 0\leq t\leq T;\\
    m(0) =h_{1}, \ \ m(t)=0, \ \ \ \ \ \ \ \ \ \ \ \ \ \ \ \ \ \ \ \ \ \ \ \ \ \ \ \ \ \ \ \ \ \ \ \ \  \ \ \ \ \ \ \ \ \ \
     T< t\leq T+\delta.
  \end{cases}
\end{equation}
Denote $\Omega_{0}:=\{ \omega\in\Omega| \xi^{*}(\omega)=0 \}$.
Now, the following necessary condition can be deduced owing to the arbitrariness of $\xi$.
\[\begin{split}
     m(T) + h_{0}\xi^{*}\geq 0, \ \ a.s. \  on \ \Omega_{0};\\
     m(T) + h_{0}\xi^{*} = 0,  \ \ a.s. \  on \ \Omega_{0}^{c},
\end{split}\]
where $m(\cdot)$ is the solution of Eq. (\ref{26}).

\subsection{Production-consumption choice optimization problem}

By applying the maximum principle established before,
we investigate a type of production-consumption choice optimization problem in this subsection.
The shape for this issue, as in \cite{Wu}, originates from Ivanov and Swishchuk \cite{Ivanov}.
For the sake of completeness, let us present the model at length.

We assume an investor is going to invest his money to invent goods,
and he could obtain benefits from the goods.
We mark the capital of investor,
the labor at time $t$ and the rate of consumption by $X(t)$, $A(t)$ and $c(t)\geq 0$, respectively.
Based upon the assumption that earning of the production is a function of the total sum of the capital and labor,
in order to depict this system,
Ramsey \cite{Ramsey} introduced the following shape:
\begin{equation}\label{27}
  \frac{dX(t)}{dt}= f(X(t),A(t)) - c(t).
\end{equation}
Since in the procedure of investment, in reality, there are some risk and delay, Chen and Wu \cite{Wu} generalized the model in (\ref{27}) to the following case:
\begin{equation}\label{28}
  \begin{cases}
   dX(t) = [f(X(t-\delta),A(t))-c(t)] dt + \sigma (X(t-\delta)) dW(t), \ \ 0\leq t\leq T;\\
    X(t) = \eta(t), \ \ \ \ \ \ \ \ \ \ \ \ \ \ \ \ \ \ \ \ \ \ \ \ \ \ \ \ \ \ \ \ \ \ \ \ \ \ \ \ \ \ \ \ \ \ \ \ \ \ \ \ \ \ \ \ \ \  \
           -\delta\leq t\leq 0,
  \end{cases}
\end{equation}
where $\eta$ is a given continuous function.

However, the rationality of the shape has been questioned as
there is no constraint for the terminal capital $X(T)$ on the basis of the hypothesis.
In fact, in real situations, sometimes the investor will set a goal (constraint) for the terminal capital $X(T)$ in the investment.
Hence we believe that in a concordant model of the production and consumption
some constraints for the terminal capital $X(T)$ should be considered, i.e.,
$X(T)\in Q$,  where $Q\in \mathbb{R}^{n}$.

Let us consider the following model which modified (\ref{27}) and (\ref{28}):
\begin{equation*}\label{29}
  \begin{cases}
   dX(t) = [f(X(t-\delta),A(t))-c(t)] dt + \sigma (X(t-\delta),c(t)) dW(t), \ \ 0\leq t\leq T;\\
    X(t) = \eta(t), \ \ \ \ \ \ \ \ \ \ \ \ \ \ \ \ \ \ \ \ \ \ \ \ \ \ \ \ \ \ \ \ \ \ \ \ \ \ \ \ \ \ \ \ \ \ \  \ \ \ \ \ \ \ \ \ \ \ \ \ \ \ \ \ \ \
           -\delta\leq t\leq 0.
  \end{cases}
\end{equation*}
For simplicity, let $n=d=1$. Consider the following hypotheses:

\begin{itemize}
  \item [(1)] The function $f(X(t-\delta),A(t))=KX^{\alpha}(t-\delta)A^{\beta}(t)$,
              where $K,\alpha,\beta$ are some suitable constants.
              Moreover, let $\alpha=\beta=1$ and $A(t)\equiv y$ be a constant.
  \item [(2)] The terminal constraint $Q\in \mathbb{R}$ is a given convex set.
\end{itemize}

Under the above assumptions, we can rewrite our shape as follows:
\begin{equation}\label{30}
  \begin{cases}
   dX(t) = [KyX(t-\delta)-c(t)] dt + \sigma (X(t-\delta),c(t)) dW(t), \ \ 0\leq t\leq T;\\
    X(t) = \eta(t),  \ \ \ \ \ \ \ \ \ \ \ \ \ \ \ \ \ \ \ \ \ \ \ \ \ \ \ \ \ \ \ \ \ \ \ \ \ \ \ \ \ \ \ \ \ \ \ \ \ \ \ \ \ \ \ \ \ \ \
           -\delta\leq t\leq 0.
  \end{cases}
\end{equation}
By electing the hypothesis rate $c(t)\geq 0$ under the terminal constraint $X(T)\in Q$,
the purpose is to maximize the following desired function:
\begin{equation}\label{31}
  J(c(\cdot))=E\left[ \int_0^T e^{-rt}\frac{c^{\gamma}(t)}{\gamma} dt + X(T) \right],
\end{equation}
where $r$ represents the bond rate, $\gamma\in (0,1)$, and $1-\gamma$ represents the investor's risk aversion.

 Clearly, this is a special case of Problem A when
 $$ \mathcal{U}_{ad} \equiv \{ c(\cdot)| c(\cdot)\in L_{\mathbb{F}}^{2}(0,T;\mathbb{R}^{+}) \}$$
with
\begin{align*}
&b(t,x,x',c)=Ky\cdot x' - c,  \ \ \sigma(t,x,x',c)=\sigma(x',c), \\
&\widetilde{l}(t,x,x',c)=-e^{-rt}\frac{c^{\gamma}}{\gamma}, \ \ \phi(x)=-x.
\end{align*}

Now, let $q\equiv \sigma(x',c)$ and $f(x',q)=-Ky\cdot x' + \widetilde{\sigma}(x',q)$, where $\widetilde{\sigma}$ is the inverse function of $\sigma$ w.r.t $c$, i.e., $c=\widetilde{\sigma}:=\widetilde{\sigma}(x',q)$.
Then one can rewrite (\ref{30}) as follows:
\begin{equation}\label{32}
  \begin{cases}
   -dX(t) = [-KyX(t-\delta)+ \widetilde{\sigma}(X(t-\delta),q(t))] dt - q(t) dW(t), \ \ 0\leq t\leq T;\\
    X(t) = \eta(t),  \ \ \ \ \ \ \ \ \ \ \ \ \ \ \ \ \ \ \ \ \ \ \ \ \ \ \ \ \ \ \ \ \ \ \  \ \ \ \ \ \ \ \ \ \ \ \ \ \ \ \ \ \ \ \ \ \ \ \ \ \ \ \ \
           -\delta\leq t< 0.
  \end{cases}
\end{equation}
Define $U=\{ \xi| E|\xi|^{2}<\infty, \ \xi\in Q, \ a.s. \}$ and consider the following performance function:
\begin{equation}\label{33}
  J(\xi)=-E\left[ \int_0^T e^{-rt}\frac{\widetilde{\sigma}^{\gamma}(t)}{\gamma} dt + \xi \right].
\end{equation}
Then problem (\ref{31}) is equivalent to the following problem:
\begin{equation}\label{34}
  \begin{cases}
    $Minimize$ \ \ \ \ J(\xi) \\
    $subject$ \ $to$ \ \ \ \xi\in U; \  X^{\xi}(0)=a.
  \end{cases}
\end{equation}
Consider the adjoint equation
\begin{equation}\label{35}
  \begin{cases}
   dm(t) = E^{\mathcal{F}_{t}}\big[\big(-Ky + (\widetilde{\sigma}_{x_{\delta}}|_{t+\delta})\big)m(t+\delta)\big] dt \\
     \ \ \ \ \ \ \ \ \ \ \ + \left[\widetilde{\sigma}_{q}(t)m(t)-h_{0}e^{-rt}\cdot\widetilde{\sigma}^{\gamma-1}(t)\right] dW(t),  \ \ 0\leq t\leq T;\\
    m(0) =h_{1}, \ \ m(t)=0, \ \ \ \ \ \ \ \ \ \ \ \ \ \ \ \ \ \ \ \ \ \ \ \ \ \ \ \ \ \ \ \ \ \ \ T< t\leq T+\delta,
  \end{cases}
\end{equation}
where $h_{1}\in \mathbb{R}$ is a parameter.
Denote $\Omega_{0}:=\{ \omega\in\Omega| \xi^{*}(\omega)=\partial Q \}$.
Therefore, by using Theorem \ref{18}, one obtains the result below.
\begin{theorem}
Suppose $(X^{*}(\cdot),c^{*}(\cdot))$ is an optimal pair to problem (\ref{31}), then we have $h_{0}\in \mathbb{R}^{+}$ and $h_{1}\in\mathbb{R}$ satisfying
$|h_{0}|+|h_{1}|\neq 0$ so that $\xi^{*}\equiv X^{*}(T)$, we have
\[\begin{split}
     m(T) + h_{0}\xi^{*}\geq 0, \ \ a.s. \  on \ \Omega_{0};\\
     m(T) + h_{0}\xi^{*} = 0,  \ \ a.s. \  on \ \Omega_{0}^{c},
\end{split}\]
where $m(\cdot)$ is the solution to Eq. (\ref{35}) with parameter $h_{1}$.
\end{theorem}

\section{Conclusions}

In this content, we study a stochastic optimal control problem for stochastic differential delayed equation
with terminal state constraint (at the terminal time, the state is constrained in a convex set).
 However, the control problem with terminal non-convex state constraint is still open.
We will focus on the open problem in the future study.

\section*{Competing interests}
The first and second authors announce to have no competing interests.

\section*{Authors' contributions}
The first and second authors contributed equally to the writing of this paper.
Moreover, the first and second authors read and approved the final manuscript.

\section*{Acknowledgements}

The first author would like to appreciate the Department of Mathematics of University of Central Florida, USA,
for its hospitality, and
express the gratitude to Prof. Qingmeng Wei for her careful reading of this paper and helpful comments.

\end{document}